\documentclass[article, 12pt]{amsart}
\usepackage{amsfonts,url}
\usepackage{color}
\usepackage[margin=23mm]{geometry}
\usepackage{amssymb, amsmath, amsthm, amscd, accents, mathtools}
\usepackage{graphicx,times,bbm,url}
\usepackage[T1]{fontenc}
\usepackage{eucal,mathrsfs,dsfont}
\usepackage{hyperref}
\newtheorem{lemma}{Lemma}
\newtheorem{theorem}[lemma]{Theorem}
\newtheorem{proposition}[lemma]{Proposition}
\theoremstyle{definition}
\newtheorem{definition}[lemma]{Definition}
\newtheorem{remark}[lemma]{Remark}

\DeclareMathOperator{\sgn}{sgn}
\newcommand*\mE{\mathbb{E}}
\newcommand*\mR{\mathbb{R}}
\newcommand*\mP{\mathbb{P}}
\newcommand*\edp{\mathcal{E}_D^p}
\newcommand*\hdp{\mathcal{H}_{\partial D}^p}
\newcommand*\vdp{\mathcal{V}^{1,p}(D)}
\newcommand*\edpt{\widetilde{\mathcal{E}}_D^p}
\newcommand*\eupt{\widetilde{\mathcal{E}}_U^p}
\newcommand*\dn{\partial_{\vec{n}}}

\newcommand{\fr}[1]{\langle #1 \rangle}
\DeclareMathOperator{\Tr}{Tr}



\title[Douglas formula]{The Douglas formula in $L^p$}
\author[K. Bogdan]{Krzysztof Bogdan}
\address{Faculty of Pure and Applied Mathematics,
Wroc\l aw University of Science and Technology,
Wyb. Wyspia\'nskiego 27, 50-370 Wroc\l aw, Poland}
\email{krzysztof.bogdan@pwr.edu.pl}
\author[D. Fafuła]{Damian Fafuła}
\address{Faculty of Pure and Applied Mathematics,
Wroc\l aw University of Science and Technology,
Wyb. Wyspia\'nskiego 27, 50-370 Wroc\l aw, Poland}
\email{damian.fafula@pwr.edu.pl}
\author[A. Rutkowski]{Artur Rutkowski}
\address{ Norwegian University of Science and Technology, H\o{}gskoleringen 1, 7034 Trondheim, Norway}
\address{ Faculty of Pure and Applied Mathematics,
Wroc\l aw University of Science and Technology,
Wyb. Wyspia\'nskiego 27, 50-370 Wroc\l aw, Poland}
\email{artur.rutkowski@pwr.edu.pl}
\thanks{The first author was supported by the Opus grant 2018/31/B/ST1/03818 from the National Science Center (Poland).
The second author was supported by the Opus grant 2017/27/B/ST1/01339 from the National Science Center (Poland). The third author was supported by the Preludium grant 2019/35/N/ST1/04450 from the National Science Center (Poland). Part of this research was carried out while the third author was an ERCIM Alain Bensoussan fellow at NTNU}
\begin{document}
\begin{abstract}
	We prove a Douglas-type identity in $L^p$ for $1<p<\infty$.
\end{abstract}
\maketitle
\section{Introduction}
The classical Douglas formula \cite{MR1501590} (see also Rad\'o \cite[(5.2)]{MR0344979} and Chen and Fukushima \cite[(5.8.4), (5.8.3)]{MR2849840})
relates the energy of the harmonic function $u$ on the unit disk $B(0,1)\subset \mR^2$ to the energy of its boundary values $g$ on the boundary of the disk, identified with the torus $[0,2\pi)$:
\begin{equation}\label{eq:classicDouglas}
\int\limits_{B(0,1)}|\nabla u(x)|^2\, dx = \frac{1}{8\pi}\int_0^{2\pi}\int_0^{2\pi}\frac{(g(\eta) - g(\xi))^2}{\sin^2((\xi - \eta)/2)}\, d\eta\, d\xi.
\end{equation}
The formula is important in the trace theory for Sobolev spaces, since the left-hand side of \eqref{eq:classicDouglas} is the classical Dirichlet integral and the right-hand side is equivalent to the Gagliardo form in $H^{1/2}(\partial B(0,1))$, the trace space for $W^{1,2}(B(0,1))$. The identity inspired important developments in the theory of Dirichlet forms; see \cite{MR2849840,MR2778606,MR1681637}.
Doob \cite[Theorem 9.2]{JD1962} generalized \eqref{eq:classicDouglas} 
to arbitrary Greenian open sets $D\subseteq \mR^d$ with $d\ge 2$.
In this paper, we propose another extension of \eqref{eq:classicDouglas}:
\begin{equation}\label{eq:pDouglasball}
\int\limits_{B(0,1)} |\nabla u(x)|^2 |u(x)|^{p-2}\,dx = \frac{1}{2 (p-1)} \int_0^{2\pi} \int_0^{2\pi} \frac{(g(\eta)^{\langle p-1 \rangle} - g(\xi)^{\langle p-1 \rangle})(g(\eta) - g(\xi))}{4\pi\sin^2((\xi-\eta)/2)}\,d\eta\, d\xi,
\end{equation}
Here and below, $p\in (1,\infty)$ and $a^{\langle \kappa\rangle} = |a|^{\kappa}\sgn(a)$ for $a,\kappa\in \mR$, in fact, we prove that for all open bounded $C^{1,1}$ sets $D\subseteq \mR^d$ with $d\geq 2$, and harmonic functions $u$ in $D$ with boundary values $g$, 
\begin{equation}\label{eq:classicDouglasp}
\int\limits_D |\nabla u(x)|^2 |u(x)|^{p-2}\, dx \!=\! \frac{1}{2(p-1)}\int\limits_{\partial D}\int\limits_{\partial D} (g(z)^{\langle p-1\rangle} - g(w)^{\langle p-1 \rangle})(g(z)-g(w))\, \gamma_D(z,w)\, dz\, dw.
\end{equation}
 Here $\gamma_D$ is the normal derivative of the Poisson kernel of $D$, defined in Lemma \ref{lem:gamma_existence} below, and $dz$, $dw$ refer to the surface measure on $\partial D$. A direct calculation shows that the kernel in \eqref{eq:pDouglasball} is the normal derivative of the Poisson kernel of the unit ball; therefore, \eqref{eq:classicDouglasp} is an extension of \eqref{eq:pDouglasball}. We refer to \eqref{eq:classicDouglasp} as $p$-Douglas identity (Douglas identity for short) and to the sides of \eqref{eq:classicDouglasp} as $p$-forms. We remark that P. Stein \cite[(4.3)]{MR1573962} obtained an early version of the $p$-Douglas identity for the unit disk \eqref{eq:classicDouglasp} under the assumption that $u\in C^2(\overline{D})$, but without the explicit form of the right-hand side, which Douglas only gave for $p=2$. A more general variant of Stein's identity can be obtained by taking $p=2$, a power function $h$, and a harmonic function $u$ in the work of Ka\l{}amajska and Choczewski \cite[(5.1)]{MR3867979}. The \textit{non-explicit} terms in \cite{MR3867979,MR1573962} have the form $\int_{\partial D} u^{\langle p-1\rangle}\partial_{\vec{n}} u$, which usually appears in the Green's formula. One of the main features of the $p$-Douglas identity is that it presents this integral in a more explicit form, seen on the right-hand side of \eqref{eq:classicDouglasp}. This contributes to a better understanding of the boundary behavior of functions in Sobolev-type spaces.
 
The precise statement of identity \eqref{eq:classicDouglasp} is given in Theorems \ref{th:main} and \ref{th:tr}. In the first of these results we assume that $g$ on $\partial D$ is given with the right-hand side of \eqref{eq:classicDouglasp} finite, we define $u$ as its Poisson integral and we establish that the left-hand side of \eqref{eq:classicDouglasp} is, in fact, equal to the right-hand side; in particular, it is finite. Therefore, this result may be thought of as an extension-type theorem. In Theorem~\ref{th:tr} we start with a harmonic function $u$ on $D$ with the left-hand side of \eqref{eq:classicDouglasp} finite and we obtain the function $g$ on $\partial D$, of which $u$ is a Poisson integral, and for which \eqref{eq:classicDouglasp} holds true.
Therefore, the result may be thought of as a trace-type theorem.


It is worth noting that \textit{formally} we have
\begin{align}\label{eq:forms}
    \int_D |\nabla u^{\langle p/2 \rangle}(x)|^2\, dx = \frac{p^2}{4} \int_D |\nabla u(x)|^2 |u(x)|^{p-2}\, dx.
\end{align}
The former integral is convenient for the study of the trace of (not necessarily harmonic) functions with this energy form finite; see Theorem \ref{th:tr2}. We stress that equality \eqref{eq:forms} should not be taken for granted in the case $1<p<2$; then we only prove it under certain assumptions of regularity of $u$. 

For nice non-harmonic functions $v$ that \textit{vanish at the boundary}, there is another formula:
\begin{align}\label{eq:lu}
    \int_D |\nabla v(x)|^2 |v(x)|^{p-2}\, dx = \frac{1}{1-p}\int_D \Delta v(x) |v(x)|^{p-2} v(x)\, dx.
\end{align}
Here, again, the case $1<p<2$ requires special attention; we refer to Metafune and Spina \cite{MR2465581} for details. In this connection, we mention the work of Seesanea and Verbitsky \cite[Theorem~3.1]{MR3881877}, who study \eqref{eq:lu} in the context of Green potentials of non-negative measures. A variant of Douglas formula with a remainder term, which we propose in Theorem~\ref{th:TDrem} below, combines \cite{MR2465581,MR3881877} and our identity \eqref{eq:classicDouglasp} for harmonic functions. However, we adopt the simplifying assumption $v\in C^2(\overline{D})$, allowing the use of Green's identity, which is not easily available in the setting of \eqref{eq:classicDouglasp}.

One of our main inspirations and tools is the Hardy--Stein identity of Bogdan, Dyda, and Luks \cite{MR3251822}, which states that for every harmonic function $u\colon D \to \mR$,
\begin{equation}\label{eq:HS}
    \sup\limits_{U\subset\subset D}\mE^x|u(X_{\tau_U})|^p - |u(x)|^p = p(p-1)\int_D G_D(x,y) |u(y)|^{p-2}|\nabla u(y)|^2\, dy,\quad x\in D.
\end{equation}
Here, $G_D$ is the Green function of $D$ and $\tau_D$ is the first exit time from $D$ of the Brownian motion $X_t$ (more detailed definitions can be found in Section \ref{sec:prelim}). Note that \eqref{eq:HS} gives a characterization of the Hardy space $H^p(D)$. Our Douglas formula is a similar characterization of harmonic functions in Sobolev-type spaces; see also Bogdan, Grzywny, Pietruska-Pałuba, and Rutkowski \cite{2020arXiv200601932B} for an analogous discussion of nonlocal operators such as the fractional Laplacian.

On a general level, the present paper deals with the classical potential theory in the $L^p$ setting. This may indicate why the usual harmonic functions have a distinguished role for the considered $p$-forms. Integral forms similar to \eqref{eq:classicDouglasp} have already proved useful for optimal Hardy identities and inequalities in $L^p$ and the contractivity of operator semigroups acting on $L^p$, see Bogdan, Jakubowski, Lenczewska, and Pietruska-Pa\l{}uba \cite[Theorem~1-3]{MR4372148} and the discussion in \cite[Subsection 1.3]{MR4372148}.
Finally, we reiterate that the nonlocal Douglas identity analogous to \eqref{eq:classicDouglasp} is proved in \cite{2020arXiv200601932B}, with completely different methods.

We note in passing that for $d=1$, all harmonic functions on the interval $D = (a,b)$ are of the form $u(x) = cx+d$, and then the following identity holds for $1<p<\infty$ (and is left to the reader):
\begin{equation}\label{eq:Douglasd1}
    \int_a^b c^2|u(x)|^{p-2}\,dx = \frac{1}{2(p-1)}(u(b)^{\langle p-1\rangle} - u(a)^{\langle p-1\rangle})(u(b) - u(a))\frac{2}{b-a}.
\end{equation}
Since the Green function of $D = (a,b)$ for $\Delta$ is given by 
\begin{equation*}
    G_{(a,b)}(x,y) = \begin{cases}
    (b-a)^{-1}(x-a)(b-y), &\textrm{ if } a<x<y<b,\\
    (b-a)^{-1}(b-x)(y-a), &\textrm{ if } a<y\leq x<b,
    \end{cases}
\end{equation*}
(see \cite[(29) in Section 2]{MR1329992}), it can be verified that \eqref{eq:Douglasd1} is an analogue of \eqref{eq:classicDouglasp}. Having discussed the case $d=1$, for the remainder of this paper we assume
that $d\ge 2$.

The article is organized as follows. In Section \ref{sec:prelim} we introduce the main notions and properties. In Section \ref{sec:ext} we prove the Douglas identity and extension theorem when the function $g$ on $\partial D$ is given. In Section \ref{sec:tr} we prove the Douglas identity and trace theorem when the harmonic function $u$ on $D$ is given. In Section~\ref{sec:mini} we study minimization properties for the $p$-forms and we give a variant of \eqref{eq:classicDouglasp} for non-harmonic functions.
\subsection*{Acknowledgements} We thank Katarzyna Pietruska-Pa\l{}uba and Agnieszka Ka\l{}amajska for helpful discussions.

\section{Preliminaries}\label{sec:prelim}
All the sets, functions, and measures considered are assumed to be Borel. For $a,b>0$, the expression $a\gtrsim b$ means that there is a number $c>0$, i.e., \textit{constant}, such that $a \geq cb$. We write $a\approx b$ if $a\gtrsim b$ and $b \gtrsim a$.
\subsection{Geometry}In the remainder of the work we assume that $p\in (1,\infty)$ and $D$ is a $C^{1,1}$ domain at scale $q>0$, that is, for each $z\in\partial D$ there exist balls $B_1 := B(c_z, q) \subset D$ and $B_2:=B(c_z',q)\subset (\overline{D})^c$, mutually tangent at $z$ (that is, $\overline{B_1}\cap \overline{B_2} = \{z\}$). For later convenience we denote
\[
r_0 = r_0(D) := \sup\{ q>0:~D \textrm{~is~} C^{1,1} \textrm{~at scale~} q\}.
\]
It is well-known (see Aikawa et al. \cite[Lemma 2.2]{MR2286038}) that this definition is equivalent to the one in which the boundary is locally isometric to the graph of a $C^{1,1}$ function. The inward normal vector at $w\in\partial D$ will be denoted by $\vec{n}_w$ (we write $\vec{n}_w = \vec{n}$ if $w$ is implied from the context). Note that $w\mapsto \vec{n}_w$ is a Lipschitz mapping, because in a local coordinate system there is a $C^{1,1}$ function $f$ such that $\vec{n} = (\nabla f,-1)$. If $u\colon D\cup\{w\}\to \mathbb{R}$ for some point $w\in \partial D$, then the inward normal derivative of $u$ at $w$ is defined as $$\dn^w u = \lim\limits_{h\to 0^+}\frac{u(w+h\vec{n}) - u(w)}{h}.$$ We also let $\delta_D(x) = d(x,\partial D)$ for $x\in \mR^d$.
\subsection{Potential theory}Let $G_D$ be the Green function and $P_D$ be the Poisson kernel of $\Delta = \sum_{i=1}^n \frac{\partial^2}{\partial x_i^2}$ for the $C^{1,1}$ open set $D$. We have $G_D(x,y)=0$ if $x\in D^c$ or $y\in D^c$ and
$$P_D(x,z) = \dn^zG_D(x,\cdot), \quad x\in D, ~z\in \partial D.$$
We also let $P_D(w,z) = 0$ if $w\in \partial D$ and $w\neq z$ and $P_D(z,z) = \infty$ for $z\in \partial D$.
The kernels satisfy the following sharp estimates: \begin{align}\label{eq:gfest}
  G_D(x,y) &\approx \bigg(1\wedge \frac{\delta_D(x)\delta_D(y)}{|x-y|^2}\bigg)|x-y|^{2-d},\quad x,y\in D, \textrm{ if } d\geq 3\\
  \label{eq:gd2} G_D(x,y) &\approx \ln\bigg(1+\frac{\delta_D(x)\delta_D(y)}{|x-y|^2}\bigg),\quad x,y\in D, \textrm{ if } d=2,
\end{align}
and
\begin{align}\label{eq:pkest}
      P_D(x,z) \approx \frac{\delta_D(x)}{|z-x|^d}, \quad x\in \overline{D}, ~z\in\partial D, \quad  d\geq 2. 
\end{align}
The formula \eqref{eq:gfest} was given by Zhao \cite{MR842803}, \eqref{eq:gd2} comes from Chung and Zhao \cite[Theorem 6.23]{MR1329992}, and the estimate \eqref{eq:pkest} can be found in the book by Krantz \cite[Chapter 8.1]{MR1846625} or derived from \eqref{eq:gfest} and \eqref{eq:gd2}, see also Bogdan \cite[(22)]{MR1741527}.

We slightly abuse the notation by using $dx$ or $dy$ for the Lebesgue measure on $\mR^d$ and $dz$ or $dw$ for the surface measure on $\partial D$. By the result of Dahlberg \cite{MR466593}, $\omega^x(dz)$, the harmonic measure of $\Delta$ for $D$, is absolutely continuous with respect to the surface measure for all $x\in D$ and 
$$\omega^x(A) = \int_A P_D(x,z)\, dz, \quad x\in D, ~A\subseteq \partial D.$$ 
To prove \eqref{eq:classicDouglasp} we employ probabilistic potential theory.
Let $X_t$ be the Brownian motion in $\mR^d$ and let
$$\tau_D = \inf\{t>0: X_t \notin D\}.$$
By $\mP^x$ and $\mE^x$ we denote the probability and the expectation for the process $X_t$ started at $x$. It is well-known since Kakutani \cite{MR14647}  that the harmonic measure is the probability distribution of $X_{\tau_D}$: $\omega^x(A) = \mP^x(X_{\tau_D}\in A)$.

As usual, $u$ is harmonic in $D$ if $u$ is $C^2$ in $D$ and $\Delta u(x) = 0$  for every $x\in D$. Harmonic functions are characterized by the mean value property or with respect to the Brownian motion \cite{MR14647}. Namely, $u$ is harmonic in $D$ if and only if for every $U\subset\subset D$ and $x\in U$ we have
$$u(x) = \mE^x u(X_{\tau_U}).$$
Here and below we write $U\subset\subset D$ if $U$ is a relatively compact subset of $D$, that is $\overline{U}$ is bounded and $\overline{U}\subset D$. 
Conversely, the Poisson integral,
$$P_D[g](x) =  \int_{\partial D} g(z) P_D(x,z)\, dz,\quad x\in D,$$
is harmonic in $D$ if absolutely convergent at one (therefore every) point $x\in D$.
\subsection{Feller kernel}
 For $z,w\in\partial D$ we define the Feller kernel as
\[
\gamma_D(z,w) = \dn^z P_D(\cdot, w).
\]
The existence of $\gamma_D$ was studied before, see, e.g., Zhao \cite[Lemma 1]{MR842803}, but we give a short proof for the completeness of the presentation.
\begin{lemma}\label{lem:gamma_existence}
The kernel $\gamma_D(z,w)$ exists for all $z,w\in\partial D$, $z\neq w$ and
\begin{equation}\label{eq:gammaest}
\gamma_D(z,w) \approx |z-w|^{-d},\quad z\neq w.
\end{equation}
\end{lemma}

\begin{proof}
Let $z,w\in\partial D$, $z\neq w$.
Since $P_D(z,w) = 0$, we only need to calculate
\[
\gamma_D(z,w) = \lim_{h\to 0^+} \frac{P_D(z+h\vec{n},w)}{h},
\]
where $\vec{n} = \vec{n}_z$.
Let $x_0\in D$ and $h>0$ be small. We have
\[
\frac{P_D(z+h\vec{n},w)}{h} = \frac{P_D(z+h\vec{n},w)}{G_D(x_0, z+h\vec{n})} \cdot \frac{G_D(x_0, z+h\vec{n})}{h},
\]
and
\[
\lim_{h\to 0^+} \frac{G_D(x_0, z+h\vec{n})}{h} =  P_D(x_0,z).
\]
The existence of the limit \begin{equation}\label{eq:bhp}
\lim_{h\to 0^+} \frac{P_D(z+h\vec{n},w)}{G_D(x_0, z+h\vec{n})},
\end{equation}
follows from the boundary Harnack principle \cite[Theorem 7.9]{JERISON198280}. 

The estimates \eqref{eq:gammaest} follow directly from \eqref{eq:pkest} and the fact that $\delta_D(z+h\vec{n}) = h$ for small $h$. 
\end{proof}

\subsection{Bregman divergence and its properties}
The expression on the right-hand side of the Douglas identity comes from symmetrization of the so-called Bregman divergence \cite{MR3495836}. Namely, for $p>1$ and $a,b\in \mR$ we define
$$F_p(a,b) = |b|^p - |a|^p - pa^{\langle p-1\rangle}(b-a).$$
Recall that $a^{\langle k \rangle} = |a|^k \sgn a$, so $F_p$ is the second order Taylor remainder of the convex function $b\mapsto |b|^p$. It is therefore an instance of Bregman divergence \cite{MR3495836} and, indeed, we have
$$\tfrac 12(F_p(a,b) + F_p(b,a)) = H_p(a,b) := \tfrac p2 (a^{\langle p-1\rangle} - b^{\langle p-1 \rangle})(a-b).$$
The following approximations hold true:
\begin{equation}\label{eq:fpequiv}
 H_p(a,b) \approx F_p(a,b) \approx (a-b)^2(|b|\vee |a|)^{p-2}\approx (a^{\fr{p/2}} -b^{\fr{p/2}})^2, \quad a,b\in\mR.\end{equation}
 The second comparison was proved in \cite[(2.19)]{MR2400106} in a multidimensional setting. The one-dimensional case was rediscovered, e.g., in \cite[Lemma 6]{MR3251822}. Optimal constants are known for some arguments (in the lower bound for $F_p$ with $p\in (1,2)$ and the upper bound for $p\in (2,\infty)$), see \cite{2022-KB-MW} and \cite[Lemma 7.4]{MR1759788}.
The first comparison in \eqref{eq:fpequiv} follows from the second one. A historical discussion of the third comparison in \eqref{eq:fpequiv} is given \cite[Subsection 1.3]{MR4372148}. 
Many special cases of \eqref{eq:fpequiv} can be found in the earlier works \cite{MR1307413,MR1103113,MR1409835}; we refer to \cite[Section 2.2]{2020arXiv200601932B} for a full proof.

\section{The Douglas identity}\label{sec:ext}
Recall that $p\in (1,\infty)$. We define the following forms:
$$\edp[u] = p(p-1)\int_D |u(x)|^{p-2}|\nabla u(x)|^2\, dx,$$
$$\hdp[g] = \int_{\partial D}\int_{\partial D} F_p(g(z),g(w))\gamma_D(z,w)\, dz\, dw.$$
By symmetrization we get
$$\hdp[g] = \frac p2 \int_{\partial D}\int_{\partial D} (g(z)^{\langle p-1\rangle} - g(w)^{\langle p-1 \rangle})(g(z)-g(w))\gamma_D(z,w)\, dz\, dw.$$
Since $F_p\geq 0$, $\hdp[g]$ is well-defined (possibly infinite) for every Borel $g\colon \partial D\to \mR$.
\begin{lemma}\label{lem:hdp}
If $\hdp[g] < \infty$, then $g\in L^p(\partial D)$ and $P_D[g](x)$ is finite for all $x\in D$.
\end{lemma}
\begin{proof}
Assume that $\hdp[g] < \infty$, so that
\[
\int_{\partial D}\int_{\partial D} F_p(g(z),g(w))\gamma_D(z,w)\, dz\, dw < \infty.
\]
We fix $w\in\partial D$, such that
\[
\int_{\partial D} F_p(g(z),g(w))\gamma_D(z,w)\, dz<\infty.
\]
From Lemma \ref{lem:gamma_existence} it follows that $\inf\limits_{z\in \partial D,\, z\neq w} \gamma_D(z,w) > 0$, so 
\[
\int_{\partial D} F_p(g(z), g(w))\,dz < \infty.
\]
From \eqref{eq:fpequiv}, for $|a|\geq 2|b|$ we get
\begin{align*}\label{Fp1}
F_p(a,b) \gtrsim  (|a| - \tfrac{1}{2} |a|)^2(|b|\vee |a|)^{p-2} = \tfrac{1}{4} |a|^p.
\end{align*}
It follows that 
\begin{align*}
 \int_{\partial D \cap \{z:~|g(z)|\geq2|g(w)|\}} |g(z)|^p\,dz \lesssim  \int_{\partial D} F_p(g(z), g(w))\,dz<\infty,
\end{align*}
therefore,
\begin{align*}
	\int_{\partial D} |g(z)|^p\,dz  &= \int_{\partial D \cap \{z:~|g(z)|< 2|g(w)|\}} |g(z)|^p\,dz + \int_{\partial D \cap \{z:~|g(z)|\geq2|g(w)|\}} |g(z)|^p\,dz \\&\leq 2^p |g(w)|^p |\partial D| + \int_{\partial D \cap \{z:~|g(z)|\geq2|g(w)|\}} |g(z)|^p\,dz < \infty,
\end{align*}
so $g\in L^p(\partial D)$. 
By Jensen's inequality and \eqref{eq:pkest}, $P_D[|g|](x) < \infty$ for every $x\in D$.
\end{proof}

Aside from both using the Hardy--Stein identity, the proofs of Theorems \ref{th:main} and \ref{th:tr} have distinct flavors. The former is mostly self-contained and avoids the abstract potential theory, in contrast to the approaches for $p=2$ in \cite{JD1962} and \cite[Chapter~5.8]{MR2849840}. The setting of Theorem~\ref{th:tr} is more delicate because we do not assume the existence of a boundary function $g$; we need to construct it using traces of harmonic functions in Hardy and Sobolev spaces.
\begin{theorem}\label{th:main}
Assume that $\hdp[g] < \infty$. Then the Douglas identity \eqref{eq:classicDouglasp} holds true, that is,
\[
\hdp[g] = \edp[P_D[g]].
\]
\end{theorem}
\begin{proof}
Let $u = P_D[g]$. The martingale convergence argument from the proof of \cite[Proposition 3.4]{2020arXiv200601932B} applies, and we get
$$\sup\limits_{U\subset\subset D}\mE^x|u(X_{\tau_U})|^p = \mE^x|g(X_{\tau_D})|^p.$$
Therefore, by using the Hardy--Stein identity \eqref{eq:HS} we find that
\begin{equation}\label{eq:HS2}
    \mE^x|g(X_{\tau_D})|^p - |u(x)|^p = p(p-1)\int_D G_D(x,y)|u(y)|^{p-2}|\nabla u(y)|^2\, dy.
\end{equation}
By the proof of Lemma \ref{lem:hdp},
\begin{equation}\label{eq:aew}
\int_{\partial D} F_p(g(w),g(z))\gamma_D(z,w)\, dz<\infty
\end{equation}
for almost every $w\in\partial D$. For such $w$, we shall compute the corresponding normal derivative of the left-hand side of \eqref{eq:HS2}. Recall that the inward normal vector at $w\in\partial D$ is denoted by $\vec{n} = \vec{n}_w$. By the ``$p$-variance'' formulas \cite[Lemma 2.1]{2020arXiv200601932B} (see also \cite[(9.4)]{JD1962}) we have
\begin{align}
  \label{eq:pvar}  \mE^x |g(X_{\tau_D})|^p - |u(x)|^p &= \int_{\partial D} F_p(u(x),g(z))P_D(x,z)\, dz\\
 \label{eq:pvar2}   &= \int_{\partial D} F_p(g(w),g(z))P_D(x,z)\, dz - F_p(g(w),u(x)).
\end{align}
Taking $x = w + h\vec{n}$ with small $h>0$ and using \eqref{eq:gammaest}, we get
$$F_p(g(w),g(z))\frac {P_D(x,z)}{h}  =F_p(g(w),g(z))\frac {P_D(x,z)}{\delta_D(x)} \lesssim F_p(g(w),g(z))\gamma_D(z,w).$$
We let $r\to 0^+$, in particular, $x\to w$. By the dominated convergence theorem and \eqref{eq:aew},
\begin{align*}
     \dn^w\int_{\partial D} F_p(g(w),g(z))P_D(\cdot,z)\, dz = \int_{\partial D} F_p(g(w),g(z))\gamma_D(w,z)\, dz.
\end{align*}
By \eqref{eq:pvar2} and the fact that $F_p\geq 0$, we get (for $z\neq w$)
\begin{equation*}
\limsup\limits_{h\to 0^+}\frac 1h\int_{\partial D} F_p(u(w+h\vec{n}),g(z)) P_D(w+h\vec{n},z)\, dz \leq \int_{\partial D} F_p(g(w),g(z))\gamma_D(w,z)\, dz.
\end{equation*}
On the other hand, \cite[Theorem 5.8]{JERISON198280} states that $u(w+h\vec{n})$ converges to $g(w)$ as $h\to 0^+$ for almost every $w\in\partial D$. For such $w$, by Fatou's lemma and Lemma \ref{lem:gamma_existence}, we find that 
\begin{equation*}\liminf\limits_{h\to 0^+}\frac 1h \int_{\partial D} F_p(u(w+h\vec{n}),g(z))P_D(w+h\vec{n},z)\, dz \geq \int_{\partial D} F_p(g(w),g(z))\gamma_D(w,z)\, dz.
\end{equation*}
Therefore, 
$$\dn^w(\mE^{\, \cdot}|g(X_{\tau_D})|^p - |u(\cdot)|^p) = \int_{\partial D} F_p(g(w),g(z))\gamma_D(w,z)\, dz.$$
From \eqref{eq:HS2} it follows that for almost every $w\in \partial D$,
\begin{align}\label{eq1}
    \int_{\partial D} F_p(g(w),g(z))\gamma_D(z,w)\, dz &=  p(p-1)\dn^w\int_D G_D(\cdot,y)|u(y)|^{p-2}|\nabla u(y)|^2\, dy.
\end{align}
By Fatou's lemma, for $w\in \partial D$,
\begin{equation*}
    \dn^w\int_D G_D(\cdot,y)|u(y)|^{p-2}|\nabla u(y)|^2\, dy \geq \int_D P_D(y,w) |u(y)|^{p-2}|\nabla u(y)|^2\, dy.
\end{equation*}
It follows that
\begin{align*}
    \int_{\partial D}\int_{\partial D} F_p(g(w),g(z))\gamma_D(z,w)\, dz\, dw &=  p(p-1)\int_{\partial D}\dn^w\int_D G_D(\cdot,y)|u(y)|^{p-2}|\nabla u(y)|^2\, dy\, dw\\
    &\geq p(p-1)\int_{\partial D} \int_D P_D(y,w) |u(y)|^{p-2}|\nabla u(y)|^2\, dy\, dw\\
    &=p(p-1)\int_D |u(y)|^{p-2}|\nabla u(y)|^2\, dy.
\end{align*}
In particular, the last expression is finite. We next show that the inequality above is actually an equality, that is, the reverse inequality holds. By Fatou's lemma and Fubini--Tonelli, 
\begin{align}
    &\int_{\partial D}\dn^w\int_D G_D(\cdot,y)|u(y)|^{p-2}|\nabla u(y)|^2\, dy\, dw\nonumber\\ \leq &\liminf\limits_{h\to 0^+} \int_{\partial D}\int_D\frac{G_D(w+h\vec{n},y)}{h}|u(y)|^{p-2}|\nabla u(y)|^2\, dy\, dw\nonumber\\
    \label{eq:dct}=&\liminf\limits_{h\to 0^+}\int_{D}|u(y)|^{p-2}|\nabla u(y)|^2\int_{\partial D} \frac{G_D(w+h\vec{n},y)}{h}\, dw\, dy.
\end{align}
With the intent of using the dominated convergence theorem in \eqref{eq:dct}, we next show that
\begin{align}\label{eq:convergence}
\lim_{h\to 0^+} \int_{\partial D} \frac{G_D(w+h\vec{n},y)}{h}\,dw = \int_{\partial D} P_D(y,w)\,dw = 1, \quad y\in D,
\end{align}
and that there exists $C>0$ such that for small $h>0$,
\begin{align}\label{eq:bounded}\int_{\partial D}\frac{G_D(w+h\vec{n},y)}{h}\, dw \leq C, \quad y\in D.
\end{align}
For the remainder of the proof, we assume that $h<(r_0/2) \wedge (1/2L)$, where $L$ is the Lipschitz constant of the mapping $w\mapsto \vec{n}_w$. For $y\in D$, $w\in \partial D$ and $h\leq\delta_D(y)/2$,
$$|w-y| \leq |w+h\vec{n} -y| + h \leq |w+h\vec{n} - y| + \frac{\delta_D(y)}{2} \leq  |w+h\vec{n} - y| + \frac{|w-y|}2,$$ hence $|w-y| \leq 2|w+h\vec{n} - y|$. Thus, by \eqref{eq:gfest} and \eqref{eq:gd2}, the inequalities
\begin{align*}
\frac{G_D(w+h\vec{n},y)}{h} \lesssim \frac{\delta_D(y)}{|w+h\vec{n} - y|^d} \lesssim \frac{\delta_D(y)}{|w-y|^d} \lesssim P_D(y,w)
\end{align*}
hold with constants independent of $y$ and $h$.
Hence, the dominated convergence theorem gives \eqref{eq:convergence} for every $y\in D$. We also get \eqref{eq:bounded} in the case $\delta_D(y) \geq 2h$. It remains to prove \eqref{eq:bounded} for $\delta_D(y) < 2h$. Assume first that $d\geq 3$. Then, by \eqref{eq:gfest},
\begin{align*}
    \int_{\partial D}\frac{G_D(w+h\vec{n},y)}{h}\, dw &\lesssim \int_{\partial D}\frac1h\bigg(1\wedge \frac{h^2}{|w+h\vec{n}-y|^2}\bigg)|w+h\vec{n} - y|^{2-d}\, dw.
\end{align*}
We claim that we may assume without loss of generality that $\delta_D(y) = h$, in the sense that the latter integrand in this case is maximal up to a multiplicative constant. Indeed, since $h<r_0/2$ and $\delta_D(y) < 2h$, there is a unique point $w_y\in \partial D$ for which $\delta_D(y) = |y - w_y|$. If we let $y^* = w_y + h\vec{n}_{w_y}$, then $\delta_D(y^*) = h$ and so our claim is true given that $$|w + h\vec{n} - y| \gtrsim |w+h\vec{n} - y^*|.$$ In order to prove this we first define
$$D_h := \{x\in D: \delta_D(x) > h\}.$$
Note that
$$\partial D_h = \{x = w+\vec{n}h: w\in \partial D\}$$
and the correspondence between $x=w+\vec{n}h$ and $w$ is one to one (this is true because $D$ is $C^{1,1}$ and the interior ball with radius smaller than $r_0$ is unique and tangent to exactly one point of the boundary).
Therefore, since $y$ lies on the line segment connecting $w_y$ and $w_y + 2h\vec{n}(w_y)$, we have $$\overline{B(y,|y-y^*|)}\cap \partial D_h=\{y^*\},$$
because if any other point was in the intersection, then it would mean that $\delta_D(y)$ is attained at two points of $\partial D$.
Consequently, for every $x\in \partial D_h$,
$$|y^*-x| \leq |y^* - y| + |y-x| \leq 2|y-x|,$$
which proves the claim. As a consequence, we see that
\begin{align*}
    &\int_{\partial D}\frac1h\bigg(1\wedge \frac{h^2}{|w+h\vec{n}-y|^2}\bigg)|w+h\vec{n} - y|^{2-d}\, dw\\ \lesssim &\int_{\partial D}\frac1h\bigg(1\wedge \frac{h^2}{|w+h\vec{n}-y^*|^2}\bigg)|w+h\vec{n} - y^*|^{2-d}\, dw.\\
\end{align*}
By the $C^{1,1}$ geometry of $D$, we find that
$$|w-w_y| \leq |w+h\vec{n}- y^*| + h|\vec{n}_{w} - \vec{n}_{w_y}| \leq |w+h\vec{n} - y^*| +hL|w-w_y|.$$ Therefore, since $h<1/2L$, we get
$$|w-w_y| \leq 2|w+h\vec{n}-y^*|,$$ so that
\begin{align*}
    \int_{\partial D}\frac1h\bigg(1\wedge \frac{h^2}{|w+h\vec{n}-y^*|^2}\bigg)|w+h\vec{n} - y^*|^{2-d}\, dw
    \lesssim &\int_{\partial D} \frac1h\bigg(1\wedge \frac{h^2}{|w-w_y|^2}\bigg)|w-w_y|^{2-d}\, dw\\
    =&\int_{|w-w_y|>h} + \int_{|w-w_y|\leq h} =: I_1 + I_2.
\end{align*}
By using polar coordinates,
 \begin{align*}
     I_1 \lesssim h\int_{\{w\in \partial D: |w-w_y| >h\}} |w-w_y|^{-d}\, dw \approx  h\int_{\{\xi\in \mR^{d-1}: h< |\xi|\leq 1\}} |\xi|^{-d}\, d\xi \approx h\int_h^1 r^{-2} \, dr \approx 1
 \end{align*}
 and
\begin{align*}
    I_2 \lesssim \frac 1h\int_{\{w\in \partial D: |w-w_y| \leq h\}} |w-w_y|^{2-d}\, dw\approx \frac 1h \int_{\{\xi\in \mR^{d-1}: 0< |\xi| \leq h\}} |\xi|^{2-d}\, d\xi \approx \frac 1h \int_0^h \, dr = 1,
 \end{align*}
 thus the case $\delta_D(y) < 2h$ is completed, and so \eqref{eq:bounded} is proven for $d\geq 3$. The case $d=2$ is obtained by similar arguments with the help of the following computation:
 $$\int_0^1\frac{1}{h}\ln\bigg(1+\frac{h^2}{t^2}\bigg)\, dt = \int_h^{\infty} \frac{\ln(1+u^2)}{u^2}\, du \leq \int_0^{\infty} \frac{\ln(1+u^2)}{u^2}\, du <\infty .$$

Then, thanks to \eqref{eq:convergence} and \eqref{eq:bounded}, we may use the dominated convergence theorem to interchange the limit and the integral in \eqref{eq:dct}, which ends the proof for $d=2$.

Note that the case $d=1$ is summarized in \eqref{eq:Douglasd1}.
\end{proof}
\section{Trace-type results}\label{sec:tr}
In the previous section the starting point for our considerations was the function $g$ on the boundary; our  results could be thought of as any extension-type theorem. In this section, we will focus on the complementary trace-type theorem. Most importantly, we prove the Douglas identity for harmonic functions $u$ on $D$, in particular we exhibit the boundary function $g$ for such $u$, see Theorem \ref{th:tr}.

\begin{definition}
We define a Sobolev-type space
$$\vdp = \{u\in L^p(D): \int_D |\nabla u^{\fr{p/2}}(x)|^2\, dx < \infty\}.$$
To clarify, the above gradient of $u^{\fr{p/2}}$ is understood in the distributional sense and it is assumed to be a square integrable function.
\end{definition}
The equality \eqref{eq:chain} in the following result demonstrates the relation of $\mathcal V^{1,p}(D)$ to the previously studied forms.
\begin{lemma}\label{lem:chain}
Assume that $p\in (1,2)$, $u\in C^2(D)$ and let $x\in D$. If either
\begin{itemize}
    \item $u(x) \neq 0$, or
    \item $u(x) = 0$ and $\nabla u(x) = 0$,
\end{itemize}
then $\nabla u^{\fr{p/2}}(x)$ exists in the classical sense and
\begin{equation}\label{eq:chain}\nabla u^{\fr{p/2}}(x) = \frac p2 \nabla u(x)|u(x)|^{p/2-1}.\end{equation}
If $p\in [2,\infty)$, then \eqref{eq:chain} holds for every $x\in D$.
\end{lemma}
\begin{proof}
The case $p\in [2,\infty)$ is trivial, so in the sequel we let $p\in(1,2)$. If $u(x) \neq 0$, then the statement follows immediately from the chain rule. If $u(x) = 0$ and $\nabla u(x) = 0$, then for $y$ close to $x$ we have $|u(y)| = |u(y) - u(x)| \lesssim |y-x|^2$, hence $$|u(y)^{\fr{p/2}}  - u(x)^{\fr{p/2}}| = |u(y)^{\fr{p/2}}| \lesssim |y-x|^p.$$
Since $p>1$, we find that $\nabla u^{\fr{p/2}}(x) = 0$, which ends the proof.
\end{proof}
\begin{lemma}\label{lem:chain2}
Assume that $u\in \vdp\cap C^{2}(D)$. Then the gradient $\nabla u^{\fr{p/2}}$ exists in the classical sense almost everywhere in $D$, coincides with the weak gradient, and \eqref{eq:chain} holds
for a.e. $x\in D$.
\end{lemma}
\begin{proof}
Since $u\in \vdp$, we have $u^{\fr{p/2}}\in W^{1,2}(D)$. By Maz'ya \cite[1.1.3, Theorem 1]{MR2777530} we get that $\nabla u^{\fr{p/2}}$ exists in the classical sense almost everywhere and coincides with the weak gradient. Formula \eqref{eq:chain} in case $p\geq 2$ now follows from the chain rule, thus in the sequel we assume that $p\in (1,2)$. Note that if $\nabla u^{\fr{p/2}}(x)$
exists in the classical sense, then either $u(x)\ne 0$ or $u(x) = 0$ and $\nabla u(x) = 0$. Indeed, let $e_i$ be the unit vector in the $i$-th coordinate direction and assume that $u(x)=0$ and $\partial_i u(x) \ne 0$. Then
$$\bigg|\frac{u^{\fr{p/2}}(x+he_i) - u^{\fr{p/2}}(x)}{h}\bigg| = \bigg|\frac{u^{\fr{p/2}}(x+he_i)}{h}\bigg|\gtrsim |h|^{p/2-1} \mathop{\longrightarrow}\limits_{h\to 0} \infty,$$ 
which contradicts the existence of $\nabla u^{\fr{p/2}}(x)$.
Thus by Lemma \ref{lem:chain} we get that \eqref{eq:chain} holds for a.e. $x\in D$.
\end{proof}

The setting of $\mathcal V^{1,p}(D)$ is convenient for formulating trace-type results, owing to the connection with the classical Sobolev spaces. Recall that the functions $u$ in $W^{1,2}(D)$ have a well-defined trace $\widetilde{\Tr}{\,u}$ which belongs to $L^2(\partial D)$, see, e.g., Evans \cite[p. 272]{MR2597943}. Here and below, the reference measure  for $L^p(\partial D)$ is the surface measure on $\partial D$.
\begin{definition}
Let $u\in \vdp$. We define the trace of $u$ as
$$\Tr u = (\widetilde{\Tr}\, u^{\fr{p/2}})^{\langle 2/p \rangle}.$$
The above expression makes sense, because $u^{\fr{p/2}} \in W^{1,2}(D)$. In consequence, $\Tr u \in L^p(\partial D)$.
\end{definition}
\begin{theorem}\label{th:tr2} Assume that $u\in L^p(D)$ satisfies
$$\int_D |\nabla u^{\fr{p/2}}(x)|^2\, dx <\infty.$$
Then the trace $g = \Tr u$ satisfies 
$$\int_{\partial D}\int_{\partial D} F_p(g(z),g(w))\gamma_D(z,w)\, dz\, dw \lesssim  \int_D |\nabla u^{\fr{p/2}}(x)|^2\, dx < \infty.$$
\end{theorem}
\begin{proof}
If $u\in L^p$, then $u^{\fr{p/2}}\in L^2(D)$. By the trace theorem for $W^{1,2}(D)$ (see, e.g., Kufner, Old\v{r}ich, and Fu\v{c}\'{\i}k \cite[Theorems~6.8.13, 6.9.2]{MR0482102}), we therefore get that the trace $g^{\fr{p/2}}$ of $u^{\fr{p/2}}$ exists, belongs to $L^2(\partial D)$ and satisfies
$$\int_{\partial D}\int_{\partial D} (g^{\fr{p/2}}(z) - g^{\fr{p/2}}(w))^2\, |z-w|^{-d}\, dz\, dw \lesssim \int_D |\nabla u^{\fr{p/2}}(x)|^2\, dx.$$
Recall that by \eqref{eq:fpequiv} we have $(a^{\fr{p/2}} - b^{\fr{p/2}})^2 \approx F_p(a,b)$. It follows that
$$\int_{\partial D}\int_{\partial D} F_p(g(z),g(w)) \gamma_D(z,w)\, dz\, dw\lesssim  \int_D |\nabla u^{\fr{p/2}}(x)|^2\, dx < \infty.$$
\end{proof}

Here is a variant of Theorem \ref{th:main} adapted to $\vdp$ spaces.

\begin{proposition}\label{prop:tr1}
Assume that $g\colon \partial D\to \mR$ satisfies
$$\int_{\partial D}\int_{\partial D} F_p(g(z),g(w))\gamma_D(z,w)\, dz\, dw < \infty.$$
Let $u=P_D[g]$. Then $\nabla u^{\fr{p/2}}(x)$ exists in the classical sense for a.e. $x\in D$ and 
$$\int_{\partial D}\int_{\partial D} F_p(g(z),g(w))\gamma_D(z,w)\, dz\, dw = \frac{4p-4}{p}\int_D |\nabla u^{\fr{p/2}}(x)|^2\, dx.$$
\end{proposition}
\begin{proof}
By virtue of Theorem \ref{th:main}, it suffices to prove that 
\begin{equation}\label{eq:Douglasy}p(p-1)\int_D |\nabla u(x)|^2 |u(x)|^{p-2}\, dx = \frac{4p-4}{p}\int_D |\nabla u^{\fr{p/2}}(x)|^2\, dx.\end{equation}
Since $u$ is harmonic, it is also smooth, so according to Lemma \ref{lem:chain}, \eqref{eq:Douglasy} obviously holds for $p\in[2,\infty)$. For $p\in (1,2)$ we will show that under present assumptions on $u$, the set
$$A = \{x\in D: u(x) = 0,\ \nabla u(x) \neq 0\}$$
has Lebesgue measure zero. Since the left-hand side of \eqref{eq:Douglasy} is finite, $|\nabla u(x)|^2|u(x)|^{p-2}$ is finite for almost all $x\in D$, but on the other hand this expression is infinite for any $x\in A$, hence $|A| = 0$ and by Lemma \ref{lem:chain} we get \eqref{eq:Douglasy} for $p\in (1,2)$.
\end{proof}
\begin{theorem}\label{th:tr}
Assume that a nontrivial harmonic function $u$ belongs to $\vdp$.
Then, for $g = \Tr[u]$ we have $u = P_D[g]$ and the p-Douglas identity holds:
\begin{align*}
\int_{\partial D}\int_{\partial D} F_p(g(z),g(w))\gamma_D(z,w)\, dz\, dw &= p(p-1)\int_D |\nabla u(x)|^2 |u(x)|^{p-2}\, dx\\ &= \frac{4p-4}{p}\int_D |\nabla u^{\fr{p/2}}(x)|^2\, dx.
\end{align*}
\end{theorem}
\begin{remark}
In the proof we use the notion of fine boundary function as a mean of identifying the trace of $u$ with the boundary function coming from the theory of Hardy spaces. It does not seem necessary or helpful to formulate the definitions of fine notions here, we refer the interested reader to the works cited in the proof.
\end{remark}
\begin{proof}[Proof of Theorem~\ref{th:tr}]
Since $u\in \vdp$, Theorem \ref{th:tr2} gives the existence of the trace $g = \Tr u$, which satisfies 
$$\int_{\partial D}\int_{\partial D} F_p(g(z),g(w))\gamma_D(z,w)\, dz\, dw < \infty.$$
Therefore, by Theorem \ref{th:main} and Lemma \ref{prop:tr1}, the Douglas identity holds for $u$ and $g$, provided that $u = P_D[g]$. Let us show that this equality is true. By Lemma \ref{lem:chain2} we have
\begin{equation*}
    \int_D |\nabla u(x)|^2 |u(x)|^{p-2}\, dx < \infty.
\end{equation*}
Fix $x_0\in D$. Since $u$ is locally bounded in $D$ and $G_D(x,x_0)$ is integrable and bounded outside any neighborhood of $x_0$, it follows that
\begin{equation*}
     \infty >  \int_D G_D(x,x_0)|\nabla u(x)|^2 |u(x)|^{p-2}\, dx \geq \sup\limits_{U\subset\subset D}\int_U G_U(x,x_0)|\nabla u(x)|^2 |u(x)|^{p-2}\, dx.
\end{equation*}
By the Hardy--Stein identity \eqref{eq:HS} we therefore obtain that
\begin{equation*}
    \sup\limits_{U\subset\subset D} \mE^x|u(X_{\tau_U})|^p < \infty.
\end{equation*}
According to Doob \cite[Lemma 4.1]{MR0084886} the above condition puts us in a position to apply \cite[Theorems 9.3 and 5.2]{MR109961} in order to get that $u$ has a fine boundary function $f$ such that
$$u(x) = P_D[f](x),\quad x\in D.$$
In order to finish the proof it suffices to show that $f=g$, which we do below. Recall that the trace in $W^{1,2}(D)$ is defined first for functions $v\in C^\infty(\overline{D})$ as the restriction $v|_{\partial D}$ and for the rest of the functions via a density argument. Consider a sequence of functions $v_n \in C^\infty(\overline{D})$ which converges to $u^{\fr{p/2}}$ in $W^{1,2}(D)$ and almost everywhere, and let $f_n$ be the fine boundary function of $v_n$ for $n=1,2,\ldots$. By the result of Hunt and Wheeden \cite[Theorem 5.7]{MR274787}, the trace and the fine boundary function agree almost everywhere for $v_n$. Using this and the definition of the trace operator in $W^{1,2}(D)$ we get
\begin{align}\label{eq:g}\|f_n - g^{\fr{p/2}}\|_{L^2(\partial D)} = \|f_n - \Tr u^{\fr{p/2}}\|_{L^2(\partial D)} = \| \Tr v_n - \Tr u^{\fr{p/2}}\|_{L^2(\partial D)} \mathop{\longrightarrow}\limits^{n\to\infty} 0.\end{align}
On the other hand, since $u^{\fr{p/2}}$ is continuous and $v_n\to u^{\fr{p/2}}$ in $W^{1,2}(D)$ (so in the BLD sense  \cite[pp.~573--574]{JD1962}), by \cite[Theorem 4.3]{JD1962} the fine boundary functions of $v_n$ converge in $L^2$ to the fine boundary function $h$ of $u^{\fr{p/2}}$, that is,
\begin{align}\label{eq:h}\|f_n - h\|_{L^2(\partial D)} \mathop{\longrightarrow}\limits^{n\to\infty} 0.\end{align}
Since the function $t\mapsto t^{\fr{p/2}}$ is continuous, we have $h = f^{\fr{p/2}}$.
Therefore, by \eqref{eq:g} and \eqref{eq:h} we conclude that $f=g$ almost everywhere on $\partial D$, which ends the proof.
\end{proof}
\section{Minimization and an identity with a remainder term}\label{sec:mini}
We define
$$\edpt[u] = \frac{4(p-1)}{p}\int_D |\nabla u^{\fr{p/2}}(x)|^2\, dx.$$
Note that formally $\edpt[u] = \edp[u]$. It is well-known that the harmonic function $P_D[g]$ minimizes the Dirichlet energy $\widetilde{\mathcal E}_D^{2}$ in $D$ among the functions satisfying $u=g$ on $\partial D$. This allows us to easily identify the minimizer of $\edpt$ under boundary condition $g$, as we do in the following proposition.
\begin{proposition}\label{prop:mini}
    Let $g\in \vdp$. Then $u=\big(P_D[g^{\fr{p/2}}]\big)^{\fr{2/p}}$ is the unique minimizer of $\edpt$ with the boundary condition $g$ in the following sense: $u^{\fr{p/2}}-g^{\fr{p/2}}\in W^{1,2}_0(D)$ and for every $v\in\vdp$ such that $v^{\fr{p/2}}-g^{\fr{p/2}}\in W^{1,2}_0(D)$, we have
$
        \edpt[u]\leq \edpt[v].
$
\end{proposition}

Due to the uniqueness, harmonic function $u=P_D[g]$ cannot be a minimizer of $\edpt$ with the boundary condition $g$ (except for $p=2$ or constant $g$). It is, however,  a quasi-minimizer.
\begin{definition}\label{d.qm}
We say that $u$ is a quasiminimizer of $\edpt$ if there exists $K\geq 1$ such that for every open $C^{1,1}$ set $U\subset\subset D$ and $v$ which agrees with $u$ on $\partial U$ we have $\eupt[u]\leq K\eupt[v]$.
\end{definition}
Quasiminimizers were introduced by Giaquinta and Giusti \cite{MR666107}. To keep the discussion below simple, in Definition~\ref{d.qm} we require the sets $U$ to be $C^{1,1}$, but we should also remark that restricting the test sets may occasionally affect the notion of the quasiminimizer, see Giusti \cite[Example~6.5]{MR1962933}.
\begin{proposition}\label{prop:qmin}
If $\hdp[g]<\infty$, then $u=P_D[g]$ is a quasiminimizer of $\edpt$.
\end{proposition}
\begin{proof}
Let $U\subset\subset D$ be $C^{1,1}$ and let $v\colon \overline{U}\to\mR^d$ be equal to $u$ on $\partial U$. We may assume that $\eupt[v]<\infty$. By the trace theorem for $W^{1,2}(U)$ (or Theorem~\ref{th:tr2} above) and \eqref{eq:fpequiv},
\begin{align*}
    \eupt[v] \gtrsim \mathcal{H}_U^2[u^{\fr{p/2}}] \approx \mathcal{H}_U^p[u].
\end{align*}
Note that since $u$ is harmonic, we have $u=P_U[u]$ in $U$, therefore by the Douglas identity in Theorem~\ref{th:tr} we get
\begin{align*}
    \mathcal{H}_U^p[u] = \eupt[u],
\end{align*}
which ends the proof.
\end{proof}
We will now give a variant of the Douglas identity for functions which need not be harmonic. 
\begin{theorem}\label{th:TDrem}
Assume that $p\in[2,\infty)$ and let $u\in C^2(\overline D)$. Then the following identities hold true 
\begin{align*}
    \edp[u] &= \edp[P_D[u]] - p\int_D \Delta u(x) u^{\langle p-1 \rangle}(x)\, dx +\frac p2 \int_D \Delta u(x) P_D[u^{\langle p-1 \rangle}](x)\, dx\\
    &= \hdp[u] -p\int_D \Delta u(x) u^{\langle p-1 \rangle}(x)\, dx +\frac p2 \int_D \Delta u(x) P_D[u^{\langle p-1 \rangle}](x)\, dx.
\end{align*}
\end{theorem}
\begin{proof}
 Let $u\in C^2(\overline{D})$. Then, since $p\in [2,\infty)$ we get that $u^{\langle p-1 \rangle} \in C^1(\overline{D})$ and $$\nabla u^{\langle p-1\rangle}(x) = (p-1)\nabla u(x) |u(x)|^{p-2},\quad x\in D.$$ This puts us in a position to use Green's identity in the following way:
 \begin{align}\label{eq:Green2}
     \int_D u^{\langle p-1 \rangle}(x)\Delta u(x)\, dx + (p-1)\int_D |\nabla u(x)|^2|u(x)|^{p-2}\, dx = -\int_{\partial D} u^{\langle p-1 \rangle}(w)\dn^w u\, dw.
 \end{align}
Let $v = P_D[u]$, $\phi = u - v$, and note that $\Delta\phi = \Delta u$ and $\phi = 0$ (and so, $u=v$) on $\partial D$. Furthermore,
\begin{align}\label{eq:normsplit}
    \dn^w u = \dn^w v + \dn^w\phi.
\end{align}
Since $u$ is $C^2(\overline D)$ we have $\Delta\phi = \Delta u =f\in C(\overline D)$ and $\phi(x) = -\frac 12\int_D G_D(x,y) f(y)\, dy$ (see, e.g., \O{}ksendal \cite[Theorem~7.4.1]{MR2001996} and \cite[page~37]{MR1329992}). Therefore, by using an argument similar to the one in \cite[Lemma~3.2.1]{Gavin} we get that
\begin{align*}
    \dn^w \phi = -\frac 12\lim\limits_{h\to 0^+}\int_D\frac{G_D(y,w+h\vec{n})}{h} f(y)\, dy = -\frac 12\int_D P_D(y,w)f(y)\, dy.
\end{align*}
Note that this means that both derivatives on the right-hand side of \eqref{eq:normsplit} exist. By Fubini's theorem,
\begin{align}
    \int_{\partial D} u^{\langle p-1 \rangle}(w)\dn^w \phi\, dw &= -\frac 12\int_{\partial D} u^{\langle p-1 \rangle}(w)\int_D P_D(y,w) f(y)\, dy \, dw\nonumber \\&= -\frac 12\int_D f(y) \int_{\partial D}u^{\langle p-1 \rangle}(w) P_D(y,w)\, dw\, dy\nonumber\\
    &=-\frac 12\int_D \Delta u(y)P_D[u^{\langle p-1\rangle}](y) dy.\label{eq:normphi}
\end{align}
By Grisvard \cite[Theorem~2.2.2.3]{MR3396210} we have $\phi\in W^{2,2}(D)$, and so $v\in W^{2,2}(D)$ as well. Since $v$ is smooth in $D$, this further yields $v^{\fr{p-1}}\in W^{1,2}(D)$. By using Green's identity \cite[Theorem~1.5.3.1]{MR3396210} and the Douglas identity of Theorem~\ref{th:main} we find that
\begin{align*}
    \int_{\partial D}u^{\langle p-1\rangle}(w)\dn^w v\, dw = \int_{\partial D}v^{\langle p-1\rangle}(w)\dn^w v\, dw &= -(p-1)\int_D|\nabla v(x)|^2|v(x)|^{p-2}\, dx\\
    &=-\frac 1p \hdp[u].
\end{align*}
Putting this together with \eqref{eq:normphi} and \eqref{eq:normsplit} we get
\begin{align*}
    -p\int_{\partial D}u^{\langle p-1\rangle}(w)\dn^w u\, dw &= \hdp[u] + \frac p2\int_D \Delta u(y)P_D[u^{\langle p-1\rangle}](y) dy\\
    &=  \edp[P_D[u]] + \frac p2\int_D \Delta u(y)P_D[u^{\langle p-1\rangle}](y) dy.
\end{align*}
By this and \eqref{eq:Green2} we obtain the desired identities.
\end{proof}

\bibliographystyle{abbrv}
\bibliography{bib_file.bib}
\end{document}